\documentclass[a4paper,10pt]{amsart}

\usepackage[english]{babel}
\usepackage[utf8]{inputenc}

\usepackage{amssymb}
\usepackage{amsmath}
\usepackage{amsthm}
\usepackage{bbm}

\newcommand{\bbC}{\mathbb{C}}
\newcommand{\bbD}{\mathbb{D}}
\newcommand{\bbN}{\mathbb{N}}

\newcommand{\bbR}{\mathbb{R}}

\newcommand{\bbZ}{\mathbb{Z}}
\newcommand{\calL}{\mathcal{L}}

\newcommand{\re}{\operatorname{Re}}

\newcommand{\per}{\operatorname{per}}
\newcommand{\pnt}{\operatorname{pnt}}

\newcommand{\phdot}{\mathord{\,\cdot\,}}

\theoremstyle{definition}
\newtheorem{definition}{Definition}[section]
\newtheorem{remark}[definition]{Remark}
\newtheorem{remarks}[definition]{Remarks}

\theoremstyle{plain}
\newtheorem{proposition}[definition]{Proposition}
\newtheorem{lemma}[definition]{Lemma}
\newtheorem{theorem}[definition]{Theorem}
\newtheorem{corollary}[definition]{Corollary}
\newtheorem{known_results}[definition]{Known Results}
\newtheorem*{open_problem}{Open Problem}

\begin{document}

\title[The peripheral spectrum of positive operators]{Growth rates and the peripheral spectrum of positive operators}
\author{Jochen Gl\"uck}
\email{jochen.glueck@uni-ulm.de}
\address{Jochen Gl\"uck, Institute of Applied Analysis, Ulm University, 89069 Ulm, Germany}
\date{\today}
\begin{abstract}
	Let $T$ be a positive operator on a complex Banach lattice. It is a long open problem whether the peripheral spectrum $\sigma_{\operatorname{per}}(T)$ of $T$ is always cyclic. We consider several growth conditions on $T$, involving its eigenvectors or its resolvent, and show that these conditions provide new sufficient criteria for the cyclicity of the peripheral spectrum of $T$. Moreover we give an alternative proof of the recent result that every (WS)-bounded positive operator has cyclic peripheral spectrum. We also consider irreducible operators $T$. If such an operator is Abel bounded, then it is known that every peripheral eigenvalue of $T$ is algebraically simple. We show that the same is true if $T$ only fulfils the weaker condition of being (WS)-bounded.
\end{abstract}

\maketitle

\section{Introduction and preliminaries}

The spectral theory of positive operators, often referred to as \emph{Perron-Frobenius theory}, is by now a classical topic in operator theory, and one of its outstanding unsolved problems is concerned with the \emph{peripheral spectrum} of an operator: Let $E$ be a complex Banach lattice and let $T: E \to E$ be a bounded linear operator with spectrum $\sigma(T)$ and spectral radius $r(T)$. We call $T$ \emph{positive} if $Tx \ge 0$ whenever $x \ge 0$. The \emph{peripheral spectrum} of $T$ is defined as
\begin{align*}
	\sigma_{\per}(T) := \{\lambda \in \sigma(T)| \, |\lambda| = r(T)\},
\end{align*}
i.e.~it consists of all spectral values with maximal modulus; the elements of $\sigma_{\per}(T)$ are called the \emph{peripheral spectral values} of $T$. Recall that a subset $M \subset \bbC$ is called \emph{cyclic} if $re^{i\theta} \in M$ ($r \ge 0$, $\theta \in \bbR$) implies $re^{ik\theta} \in M$ for all integers $k \in \bbZ$. The unsolved question mentioned above reads as follows:

\begin{open_problem}[Cyclicity Problem]
	Does every positive linear operator on a complex Banach lattice have cyclic peripheral spectrum?
\end{open_problem}

In finite dimensions this question dates back to Perron and Frobenius who proved at the beginning of the 20th century that every positive matrix has cyclic peripheral spectrum. In the 1960s and early 1970s a lot of research was done to understand the problem on infinite dimensional Banach lattices, culminating in some striking results due to Lotz, Krieger and Scheffold around 1970 (see Section~\ref{section:overview-cycl-problem} for detailed references). In particular it was shown by Lotz~\cite[Theorem~4.7]{Lotz1968} and, independently, by Krieger \cite[Folgerung~2.2.1(b)]{Krieger1969} that the peripheral spectrum of a positive operator $T$ is cyclic whenever $T$ is \emph{Abel bounded}, meaning that 
\begin{align*}
	\sup_{r > r(T)} \big(r-r(T)\big) \|R(r,T)\| < \infty;
\end{align*}
here, $R(r,T) := (r-T)^{-1}$ denotes the \emph{resolvent} of $T$ at $r$. Using a certain reduction technique, one can infer from this result that every compact positive operator has cyclic peripheral spectrum; see Section~\ref{section:overview-cycl-problem} for details. Several other sufficient conditions for the peripheral spectrum to be cyclic have been proved, too (some of them quite recently), and many of them impose some kind of growth rate on the operator or its resolvent (see again Section~\ref{section:overview-cycl-problem}). Nevertheless, the cyclicity problem is still not completely solved.

In this article we continue the research on this problem, again focussing on certain growth conditions. Our two main results are Theorems~\ref{thm:growth-rate-for-eigenvector} and~\ref{thm:quadratic-growth-and-linear-growth}. While the second theorem relies on resolvent growth estimates, the first theorem offers a somewhat different flavour. It reads as follows (for undefined notation we refer to the end of the introduction):

\begin{theorem} \label{thm:growth-rate-for-eigenvector}
	Let $E$ be a complex Banach lattice and let $T \in \calL(E)$ be positive, $r(T) = 1$. Suppose that $\lambda \in \sigma_{\per}(T)$ is an eigenvalue of $T$ and that $z \in E$ is a corresponding eigenvector with norm $\|z\| = 1$. If a sequence $1 < r_n \downarrow 1$ is given, then the following holds:
	\begin{itemize}
		\item[(a)] As $n \to \infty$, the norm of $R(r_n,T)|z|$ grows at least as fast as $\frac{1}{r_n-1}$ and at most as fast as $\|R(r_n,T)\|$; more precisely, we have $\frac{1}{r_n-1} \le \|R(r_n,T)|z|\| \le \|R(r_n,T)\|$ for each index $n$.
		\item[(b)] Suppose that the growth rate of $\|R(r_n,T)|z|\|$ is the smallest possible, i.e.~that we have $\|R(r_n,T)|z|\| \sim \frac{1}{r_n-1}$. Then $\lambda^k \in \sigma(T)$ for all $k \in \bbZ$.
		\item[(c)] Suppose that the growth rate of $\|R(r_n,T)|z|\|$ is the highest possible, i.e.~that we have $\|R(r_n,T)|z|\| \sim \|R(r_n,T)\|$. Then $\lambda^k \in \sigma(T)$ for all $k \in \bbZ$.
	\end{itemize}
\end{theorem}

At first glance, the theorem only makes an assertion about eigenvalues in the peripheral spectrum of $T$; however, this is not really a restriction since by lifting $T$ to an ultra power of $E$ one can always achieve that the peripheral spectrum entirely consists of eigenvalues. What makes Theorem~\ref{thm:growth-rate-for-eigenvector} rather interesting are the two extreme cases in assertions (b) and (c): no matter whether the norm of $R(r_n,T)|z|$ grows as slowly or as fast as possible, we can always conclude that $\lambda^k \in \sigma(T)$ for all $k \in \bbZ$. It is thus very tempting to suspect that one can generalize the result to any growth rate of $\|R(r_n,T)|z|\|$, thereby proving that the cyclicity problem has a positive answer; this conjecture is also supported by the similarity of the proofs of assertions (b) and (c) which will become apparent in Section~\ref{section:growth-rates-for-eigenvectors}.

Besides proving our main results, we want to add some new tools to the literature and explain different methods which could be useful when engaging the cyclicity problem. Therefore we give two alternative proofs for several of our results and we also provide new proofs and/or slight generalizations for some results from the literature. 

Let us briefly outline how the paper is organised: Section~\ref{section:overview-cycl-problem} gives an overview of what is already known about the cyclicity problem. Section~\ref{section:growth-rates-for-eigenvectors} contains the proof of Theorem~\ref{thm:growth-rate-for-eigenvector} and some additional information on the peripheral spectrum of positive operators which is derived from recent results in \cite{Gluck2015}. Section~\ref{section:quadratic-resolvent-growth-the-peripheral-spectrum} contains our second main result Theorem~\ref{thm:quadratic-growth-and-linear-growth}. In Section~\ref{section:invariant-ideals-for-ws-bounded-operators} we proof a proposition on invariant ideals of so-called \emph{WS-bounded} operators (see Definition~\ref{def:ws-bounded}) which is then used to derive the following two results: first we give an alternative proof for the main result of \cite{Gluck2015}, stating the every \emph{(WS)-bounded} positive operator has cyclic peripheral spectrum; second we consider (WS)-bounded irreducible operators and show several results for their peripheral (point) spectrum which were previously only known for Abel bounded irreducible operators. The paper concludes with a very short appendix summarizing some known facts about the spectrum of operators which admit an invariant subspace; those results are needed throughout the article.

It is important to note that the cyclicity problem is not an isolated topic in operator theory. For an overview of several further aspects of Perron-Frobenius theory we refer the reader to the interesting survey article by Grobler \cite{Grobler1995}. Moreover, we find it worthwhile pointing out the following related developments: After a certain climax in the development of infinite-dimensional Perron-Frobenius theory around 1970, many results for positive operators were transferred to the theory of positive $C_0$-semigroups in the 1980s (see \cite{Arendt1986} for an overview). Quite recently a spectral theory for \emph{eventually positive} operators and $C_0$-semigroups started to emerge (see e.g.~\cite{Zaslavsky2003, Johnson2004, Noutsos2006} and the references therein for the matrix case and \cite{Daners2016, DanersInPrep} and the references therein for the $C_0$-semigroup case). Research on the (peripheral) spectrum of positive operators can also be pursued in some other directions, for example concerning essential spectra (see e.g.~\cite{Caselles1987}; see also \cite{Alekhno2010} and the references therein), concerning domination (see e.g.~\cite{Raebiger1997, Raebiger2000}), or within the framework of ordered Banach algebras (see e.g.~\cite[Section~4]{Scheffold1971} and \cite[Part~D]{Arendt1986}; see also \cite{Mouton2013, Muzundu2012} and the references therein for a more abstract approach). Moreover, many results from Perron-Frobenius theory remain true if the spectrum of a positive operator is replaced with its numerical range (see e.g.~\cite{Li2002, Maroulas2002, Aretaki2013} and the references therein for the matrix case and \cite{Radl2013, Radl2015} for the infinite dimensional case).

We conclude the introduction by briefly introducing some notation. Throughout the article we assume the reader to be familiar with the theory of (real and complex) Banach lattices; standard references for this topic are for example \cite{Schaefer1974} and \cite{Meyer-Nieberg1991}. Besides the spectral theoretic notations which were already introduced above, we use the following conventions: We set $\bbN := \{1,2,3,...\}$ and $\bbN_0 := \{0\} \cup \bbN$. If $E$ is a real or complex Banach space, then $\calL(E)$ denotes the space of all bounded linear operators on $E$. If $E$ is a complex Banach space, then for every $T \in \calL(E)$ the set $\sigma_{\pnt}(T)$ denotes the \emph{point spectrum} of $T$. The intersection $\sigma_{\per}(T) \cap \sigma_{\pnt}(T)$ is called the \emph{peripheral point spectrum} of $T$ and its elements are called \emph{peripheral eigenvalues} of $T$. The dual space of a real or complex Banach space $E$ is denoted by $E'$, and the adjoint of an operator $T \in \calL(E)$ is denoted by $T' \in \calL(E')$. If $E$ is a real Banach lattice and $x,y \in E$ then we write $x < y$ to say that $x \le y$, but $x \not= y$. If $E$ is a complex Banach lattice and $x,y \in E$ then we write $x \le y$ (or $x < y$) as a shorthand to say that $x,y$ are contained in the real part of $E$ and that $x \le y$ (or $x < y$). An element $x$ of a real or complex Banach lattice $E$ is called \emph{positive} if $x \ge 0$ and the set $E_+ := \{x \in E| \, x \ge 0\}$ is called the \emph{positive cone} of $E$. An operator $T \in \calL(E)$ on a real or complex Banach lattice $E$ is called \emph{positive}, which we denote by $T \ge 0$, if $TE_+ \subset E_+$. It is well-known that the dual space $E'$ of a real or complex Banach lattice is again a real or complex Banach lattice, where $E'_+ = \{x' \in E'| \, \langle x',x \rangle \ge 0 \; \forall x \in E_+\}$; the elements of $E'_+$ are called \emph{positive functionals} on $E$. For every positive element $x$ of a real or complex Banach lattice $E$ the set $E_x := \{y \in E| \, \exists c > 0 : |y| \le cx\}$ is called the \emph{principal ideal} generated by $x$; the element $x \in E_+$ is called a \emph{quasi-interior point} of $E_+$ if $E_x$ is dense in $E$. 

A closed vector subspace $F$ of a real Banach lattice $E$ is called a \emph{sublattice} of $E$ if $|f| \in F$ for all $f \in F$. A closed vector subspace $F$ of a complex Banach lattice $E$ is called a \emph{sublattice} of $E$ if $F$ is conjugation invariant and if we have $|f| \in F$ for all $f \in F$; one can check that this is equivalent to $F$ being conjugation invariant and the real part of $F$ being a sublattice of the real part of $E$.

If $K \not= \emptyset$ is a compact Hausdorff space, then we denote by $C(K;\bbR)$ and $C(K;\bbC)$ the spaces of all real or complex-valued continuous functions on $K$, respectively; those spaces are always endowed with the supremum norm and $C(K;\bbR)$ is always equipped with the canonical order. By $\mathbbm{1}_K$ we denote the constant function on $K$ with value $1$. A linear operator $T: C(K;\bbR) \to C(K;\bbR)$ is called a \emph{Markov operator} if $T$ is positive and if $T\mathbbm{1}_K = \mathbbm{1}_K$; the same notion is also used for operators $C(K;\bbC) \to C(K;\bbC)$. 

Throughout the paper we make use of the $O$- and $o$-notation to compare the asymptotic behaviour of sequences of non-negative real numbers. For two sequences $(a_n)$ and $(b_n)$ of non-negative real numbers we write $a_n \sim b_n$ if $a_n = O(b_n)$ and $b_n = O(a_n)$.

\section{What is known about the cyclicity problem?} \label{section:overview-cycl-problem}

Many results have already been proved on the peripheral spectrum of positive operators. In this section we give a brief overview of them. For the sake of simplicity we only consider operators with spectral radius $1$; this is, of course, no loss of generality since we can always reduce the general case to this situation by a rescaling argument. 

Many known results on the topic are related to the growth behaviour of the resolvent close to the peripheral spectrum. For an intelligent reading of them the following observation is important: Let $E$ be a complex Banach lattice, let $T \in \calL(E)$ be positive with $r(T) = 1$ and let $\lambda$ be a peripheral spectral value of $T$. Then the estimate
\begin{align}
	\frac{1}{r-1} \le \|R(r\lambda,T)\| \le \|R(r,T)\| \label{form_canonical_resolvent_estimate}
\end{align}
holds for every $r > 1$. This shows that for $r \downarrow 1$ the operator $R(r,T)$ grows at least as fast in norm as the operator $R(r\lambda,T)$, and this latter operator grows at least as fast as $\frac{1}{r-1}$. One should keep this observation in mind when reading the following results which we quote from the literature:

\begin{known_results} \label{ov:cyclcity-problem-growth-conditions-on-resolvent}
	Let $E$ be a complex Banach lattice, let $T \in \calL(E)$ be positive and assume that $r(T) = 1$. Moreover, choose a sequence $1 < r_n \downarrow 1$.
	\begin{itemize}
		\item[(a)] If $\|R(r_n,T)\|$ grows as slowly as possible, i.e.~if we have $\|R(r_n,T)\| \sim \frac{1}{r_n-1}$, then the peripheral spectrum of $T$ is cyclic.
		\item[(b)] Let $\lambda \in \sigma_{\per}(T)$. If $R(r_n\lambda,T)$ grows as fast as possible, i.e.\ if we have $\|R(r_n\lambda,T)\| \sim \|R(r_n,T)\|$, then $\lambda^k \in \sigma_{\per}(T)$ for every $k \in \bbZ$.
		\item[(c)] Assume that $R(r_n,T)$ grows at most quadratically, i.e.~assume that we have $\|R(r_n,T)\| = O(\frac{1}{(r_n-1)^2})$. Let $\lambda \in \sigma_{\per}(T)$ be an eigenvalue of $T$ with eigenvector $z$ and assume that there exists a functional $x' \in E'$ which fulfils $0 \le x' \le T'x'$ and $\langle x', |z| \rangle \not= 0$. Then $\lambda^k \in \sigma_{\per}(T)$ for each $k \in \bbZ$.
	\end{itemize}
\end{known_results}

Assertion (a) was proved by Lotz \cite[Theorem~4.7]{Lotz1968} and, independently, by Krieger \cite[Folgerung~2.2.1(b)]{Krieger1969} (actually, they both assumed that $\|R(r,T)\| \sim \frac{1}{r-1}$ as $r \downarrow 1$, but a short glance at the proofs shows that one only needs to consider a sequence); the proofs of Lotz and Krieger are similar in some, but not in all aspects. In \cite[first part of the proof of Theorem~V.4.9]{Schaefer1974} a version of the proof can be read in English. We point out that assertion (a) can also be seen as a special case of assertion~\ref{ov:cyclcity-problem-growth-conditions-on-resolvent}(b) and as a special case of assertion~\ref{ov:cyclicity-problem-other-conditions}(f) below. Assertion (b) is due to Krieger \cite[Folgerung~2.2.3]{Krieger1969}; see also \cite[Theorem~7.8]{Gluck2015} for an English presentation of the result and a partly different proof. Assertion (c) is also due to Krieger who proved it in \cite[Folgerung~2.2.4]{Krieger1969} (actually, Krieger proved the result under slightly different assumptions, but it is not difficult to see that his proof also works under the assumptions made in assertion (c)). Note that the assumed inequality $x' \le T'x'$ in assertion (c) is converse to the inequality $x' \ge T'x'$ which sometimes appears in auxiliary results in Perron-Frobenius theory (see e.g.~\cite[Proposition~V.5.1]{Schaefer1974}).

Other known results on the cyclicity problem replace the growth conditions on the resolvent by different assumptions. We also give an overview of those results:

\begin{known_results} \label{ov:cyclicity-problem-other-conditions}
	Let $E$ be a complex Banach lattice, let $T \in \calL(E)$ be positive and assume that $r(T) = 1$.
	\begin{itemize}
		\item[(a)] If the operator $T$ is \emph{Abel solvable} (see \cite[p.\ 152]{Grobler1995} for a definition), then the peripheral spectrum of $T$ is cyclic.
		\item[(b)] If the spectral radius $r(T) = 1$ is a pole of the resolvent $R(\cdot,T)$ then the peripheral spectrum of $T$ is cyclic (note that $r(T)$ is automatically a pole of $R(\cdot,T)$ if $T$, or more generally some power of $T$, is compact).
		\item[(c)] If the spectral radius $r(T) = 1$ is an isolated point in the spectrum $\sigma(T)$ and if $[0,1)$ is contained in the resolvent set of $T$, then the peripheral spectrum of $T$ is cyclic.
		\item[(d)] If $\sigma(T)$ is contained in the complex unit circle, then $\sigma(T) = \sigma_{\per}(T)$ is cyclic.
		\item[(e)] If there exists a sequence $(a_n)_{n \in \bbN}$ of non-negative real numbers such that we have $T^n \ge a_n I$ for each $n \in \bbN$ (where $I$ denotes the identity map on $E$) and such that $\limsup a_n^{1/n} = 1$, then the peripheral spectrum of $T$ is cyclic.
		\item[(f)] If $T$ is (WS)-bounded (see Definition~\ref{def:ws-bounded} below or \cite[Definition~4.7 and Example~4.8]{Gluck2015}) then the peripheral spectrum of $T$ is cyclic.
	\end{itemize}
\end{known_results}

Assertion (b) is a consequence of (a) and both assertions are due to Lotz \cite[Theorems~4.9 and~4.10]{Lotz1968}; see also \cite[the Theorem~V.4.9 and its Corollary]{Schaefer1974} for an English version. Note Lotz \cite[p.\,26]{Lotz1968} uses the term \emph{aufl\"osbar} (which is German for \emph{solvable}) and Schaefer \cite[Definition~V.4.7]{Schaefer1974} uses the the term \emph{(G)-solvable} instead of ``Abel solvable'', but they all mean the same condition. Assertion (c) was proved by Krieger in \cite[Satz~2.2.3]{Krieger1969} (see also \cite[p.\,352]{Schaefer1974} where this result is stated in English). Assertions (d) and (e) are due to Zhang (see \cite[p.\,118]{Zhang1993}; there it is also noted that assertions (c) and (d) can be seen as special cases of (e)) and assertion (f) was proved by the author \cite[Theorem~7.1]{Gluck2015}. In Section~\ref{section:invariant-ideals-for-ws-bounded-operators} of the current article we will present an alternative proof for assertion (f) (see Corollary~\ref{cor:ws-bounded-positive-operators-has-cyclic-peripheral-spectrum} for details). Assertion (f) also generalizes a result of Scheffold~\cite[Satz~3.6]{Scheffold1971} asserting that every partially power-bounded positive operator with spectral radius $1$ has cyclic peripheral spectrum.

A few further conditions for the cyclicity of the peripheral spectrum were recently given by the author in~\cite[Theorem~7.4 and Corollaries~7.5 and~7.6]{Gluck2015}. We did not mention those results explicitly above since we are going to prove a slightly more general theorem which contains these three results as special cases in Theorem~\ref{thm:dae} below.

\section{Growth rates for eigenvectors and the peripheral spectrum} \label{section:growth-rates-for-eigenvectors}

We start this section with the following observation on eigenvalues of positive operators.

\begin{proposition} \label{prop:dominated-eigenvector}
	Let $E$ be a complex Banach lattice, let $T \in \calL(E)$ be a positive operator and let $\lambda \in \bbC$ be an eigenvalue of $T$ with $|\lambda| = 1$ and with corresponding eigenvector $z$. Assume that at least one of the following two assumptions is fulfilled:
	\begin{itemize}
		\item[(a)] $\sup_{n \in \bbN_0} \|T^n|z|\| < \infty$. 
		\item[(b)] $1$ is an eigenvalue of $T$ and possesses an eigenvector $x$ which fulfils $x \ge |z|$.
	\end{itemize}
	Then $\lambda^k \in \sigma_{\pnt}(T'') \subset \sigma(T)$ for each $k \in \bbZ$.
	\begin{proof}
		(a) By means of evaluation we consider $E$ as a sublattice of $E''$. Note that $T|z| \ge |Tz| = |z|$ and hence the sequence $(T^n|z|)_{n \in \bbN_0} = ((T'')^n|z|)_{n \in \bbN_0}$ is increasing. By assumption (a) the sequence thus weak${}^*$-converges to an element $x \ge |z|$ in $E''$. Since $T''$ is weak${}^*$-weak${}^*$-continuous, $x$ is a fixed point of $T''$ and 
		hence it follows from \cite[Theorem~3.2]{Gluck2015} that $\lambda^k$ is an eigenvalue of $T''$ for each $k \in \bbZ$. This proves the assertion. 
		
		(b) If condition (b) is fulfilled, then obviously condition (a) is fulfilled as well. Hence, the assertion follows.
	\end{proof}
\end{proposition}

Note that in the above proposition we did \emph{not} assume that $r(T) = 1$. Yet, throughout the paper we only need the special case of Proposition~\ref{prop:dominated-eigenvector} where $r(T) = 1$; moreover, we won't need that $\lambda^k$ is an eigenvalue of the bi-adjoint of $T$, but only that $\lambda^k \in \sigma(T)$. Instead of referring to \cite[Theorem~3.2]{Gluck2015} in the above proof of Proposition~\ref{prop:dominated-eigenvector} we could also have inferred this slightly weaker conclusion from a result of Krieger \cite[Satz~2.2.2]{Krieger1969}. We should also point out that the conclusion of Proposition~\ref{prop:dominated-eigenvector} can be considerably strengthened if we even have the equality $x = |z|$ in (b): in this case we even obtain $\lambda^k \in \sigma_{\pnt}(T)$ for each $k \in \bbZ$. This is a classical result in Perron-Frobenius theory which can be proved similarly as \cite[Theorem~C-III-2.2]{Arendt1986}. 

We now want to give two proofs of our first main result, Theorem~\ref{thm:growth-rate-for-eigenvector}. The first one is rather short, but it does not give much insight into the relation between assertions (b) and (c) in the theorem. The second proof is a bit lengthier, but has the advantage that the arguments for assertions (b) and (c) look rather similar. Hence, we find the second proof more likely to allow for a possible generalization to all growth rates of $R(r_n,T)|z|$ (although the author has not been able to prove such a generalization, yet).

\begin{proof}[First proof of Theorem~\ref{thm:growth-rate-for-eigenvector}]
	(a) For each $r > 1$ we have
	\begin{align*}
		\frac{1}{r-1} = \|R(r\lambda,T)z\| \le \|R(r,T)|z|\| \le \|R(r,T)\|.
	\end{align*}
	(b) Since $0 \le T|z| \le T^2|z| \le ...$ and since $M := \sup_{n}(r_n-1)\|R(r_n,T)|z| \| < \infty$, it follows that the sequence $(T^k|z|)_{k \in \bbN_0}$ is bounded in norm. Indeed, for each $k \in \bbN_0$ and each index $n$ we have
	\begin{align*}
		(r_n-1)R(r_n,T)|z| \ge (r_n-1) \sum_{j=k}^\infty \frac{T^j|z|}{r_n^{j+1}} \ge (r_n-1) \sum_{j=k}^\infty \frac{T^k|z|}{r_n^{j+1}} = \frac{T^k|z|}{r_n^k}.
	\end{align*}
	Hence, $\|T^k|z|\| \le M r_n^k$ for each $k \in \bbN_0$ and each index $n$. Letting $n \to \infty$ we obtain $\|T^k|z|\| \le M$ for each $k \in \bbN_0$ as claimed. This argument is taken from \cite[the proof of Theorem~3.5]{Scheffold1971}; in a more general setting, such a result can also be found in \cite[Lemma~4.13]{Gluck2015}. The assertion now follows from Proposition~\ref{prop:dominated-eigenvector}(a).
	
	(c) Consider the subset $J := \{x \in E| \, \frac{R(r_n,T)}{\|R(r_n,T)\|}|x| \to 0\}$ of $E$. This is a closed ideal in $E$ and according to the assumption, the vector $z + J$ is a non-zero element of the quotient space $E/J$. The operator $T$ leaves $J$ invariant and if $T_/$ denotes the operator induced by $T$ on $E/J$, then $z+J$ is an eigenvector of $T_/$ for the eigenvalue $\lambda$. Moreover, one easily verifies that $|z| + J$ is an eigenvector of $T_/$ for the eigenvalue $1$ (simply use the facts $T|z| \ge |z|$ and $\|R(r_n,T)\| \to \infty$ to check that $T|z| - |z| \in J$). Proposition~\ref{prop:dominated-eigenvector}(b) thus shows that $\lambda^k \in \sigma(T_/)$ for each $k \in \bbZ$, and Proposition~\ref{prop:spectrum-of-induced-operators}(c) finally implies that $\lambda^k \in \sigma(T)$ for each $k \in \bbZ$.
\end{proof}

For the second proof of Theorem~\ref{thm:growth-rate-for-eigenvector} we need to recall a few notions: Let $E$ be a complex Banach space and $T \in \calL(E)$. A complex number $\lambda$ is called an \emph{approximate eigenvalue} of $T$ if there is a sequence $(x_n) \subset E$ such that $0 < \liminf_n \|x_n\| \le \limsup_n \|x_n\| < \infty$ and such that $(\lambda - T)x_n \to 0$. In this case the sequence $x_n$ is called an \emph{approximate eigenvector} of $T$ for the approximate eigenvalue $\lambda$. By $l^\infty(\bbN;E)$ we denote the space of all norm bounded sequences $x = (x_n)_{n \in \bbN}$ in $E$, endowed with the supremum norm $\|x\|_\infty := \sup_{n\in \bbN} \|x_n\|$, and by $c_0(\bbN;E)$ we denote the closed subspace of $l^\infty(\bbN;E)$ which consists of all sequences in $E$ that converge to $0$. If $E$ is a Banach lattice, then so is $l^\infty(\bbN;E)$, and $c_0(\bbN;E)$ is then a closed ideal in $l^\infty(\bbN;E)$.

\begin{proof}[Second proof of Theorem~\ref{thm:growth-rate-for-eigenvector}~(b) and~(c)]
	(b) Let $\tilde T$ be the operator induced by $T$ on the complex Banach lattice $l^\infty(\bbN;E)$; it has the same spectrum as $T$. The subset $I := c_0(\bbN;E) \subset l^\infty(\bbN;E)$ is a closed ideal in $l^\infty(\bbN;E)$ and invariant under $\tilde T$; by $\hat T$ we denote the operator induced by $\tilde T$ on the quotient space $\hat E:= l^\infty(\bbN;E)/I$. 
	
	By assumption the sequence $((r_n-1)R(r_n,T)|z|)_{n \in \bbN} \subset E$ is norm bounded and thus contained in $l^\infty(\bbN;E)$; we denote its equivalence class in $\hat E$ by $\hat x$. Then $\hat x$ is non-zero and it is an eigenvector of $\hat T$ for the eigenvalue $1$ since a short computation shows that that sequence $((r_n-1)R(r_n,T)|z|)_{n \in \bbN}$ is an approximate eigenvector of $T$ for the spectral value $1$. Finally, observe that 
	\begin{align*}
		|z| = |(r_n-1)R(r_n\lambda,T)z| \le (r_n-1) R(r_n,T)|z|.
	\end{align*}
	Hence, we have $|\hat z| \le \hat x$, where $\hat z \in \hat E$ is the equivalence class of the constant sequence $(z)_{n \in \bbN}$ modulo $I$. Since $\hat z$ is an eigenvector of $\hat T$ for the eigenvalue $\lambda$ we obtain $\lambda^k \in \sigma(\hat T)$ for each $k \in \bbZ$ due to Proposition~\ref{prop:dominated-eigenvector}(b) and thus $\lambda^k \in \sigma(\tilde T) = \sigma(T)$ for every $k \in \bbZ$ by Proposition~\ref{prop:spectrum-of-induced-operators}(c).
	
	(c) Again, let $\tilde T$ be the operator induced by $T$ on $l^\infty(\bbN;E)$, but this time consider the subset $I \subset l^\infty(\bbN;E)$ which is given by
	\begin{align*}
		I := \{(x_n)_{n \in \bbN} \in l^\infty(\bbN;E)| \, \frac{R(r_n,T)}{\|R(r_n,T)\|}|x_n| \to 0\}.
	\end{align*}
	This is easily seen to be a closed, $\tilde T$-invariant ideal in the complex Banach lattice $l^\infty(\bbN;E)$. The operator induced by $\tilde T$ on the quotient space $\hat E := l^\infty(\bbN;E)/I$ is again denoted by $\hat T$. 
	
	As above, let $\hat z \in \hat E$ be the equivalence class of the constant sequence $(z)_{n \in \bbN}$ modulo $I$. It follows from the assumption of (c) that $\hat z$ is non-zero and thus it is an eigenvector of $\hat T$ for the eigenvalue $\lambda$. On the other hand, one can readily check that $\frac{R(r_n,T)}{\|R(r_n,T)\|}|T|z| - |z|\,| \to 0$; this implies that $\hat T |\hat z| = |\hat z|$, i.e.~$|\hat z|$ is an eigenvector of $\hat T$ for the eigenvalue $1$. Thus it follows from Proposition~\ref{prop:dominated-eigenvector}(b) that $\lambda^k \in \sigma(\hat T)$ for all $k \in \bbZ$. Proposition~\ref{prop:spectrum-of-induced-operators}(c) now shows that $\lambda^k \in \sigma(\tilde T) = \sigma(T)$ for each $k \in \bbZ$.
\end{proof}

The second proof of Theorem~\ref{thm:growth-rate-for-eigenvector}(b) was based on comparing an eigenvector of $T$ for the eigenvalue $\lambda$ with an approximate eigenvector for the spectral value $1$. More generally, we can also compare two approximate eigenvectors; this is the basic idea of the following definition, which we quote from \cite[Definition~7.3]{Gluck2015}:

\begin{definition} \label{def:dae}
	Let $E$ be a complex Banach lattice, let $T \in \calL(E)$ be positive and let $\lambda \not= 0$ be an approximate eigenvalue of $T$. Then $\lambda$ is said to fulfil the \emph{dominated approximate eigenvector condition} for $T$ if $|\lambda|$ is also an approximate eigenvalue of $T$ and if there exist approximate eigenvectors $(z_n)_{n \in \bbN}$ and $(x_n)_{n \in \bbN}$ for $\lambda$ and $|\lambda|$, respectively, which satisfy the estimate $|z_n| \le x_n$ for all $n \in \bbN$.
\end{definition}

In \cite[Theorem~7.4 and Corollaries~7.5 and~7.6]{Gluck2015} it was proved that if an approximate eigenvalue $re^{i\theta}$ ($r > 0$, $\theta \in \bbR$) of a positive operator $T \in \calL(E)$ fulfils the dominated approximate eigenvector condition and if an additional regularity condition on $T$ or $E$ is satisfied, then $re^{ik\theta} \in \sigma(T)$ for all $k \in \bbZ$. Using Proposition~\ref{prop:dominated-eigenvector} we can slightly generalize those results by showing that in fact no additional regularity condition is needed:

\begin{theorem} \label{thm:dae}
	Let $E$ be a complex Banach lattice, let $T \in \calL(E)$ be positive and let $re^{i\theta}$ ($r > 0$, $\theta \in \bbR$) be an approximate eigenvalue of $T$ which fulfils the dominated approximate eigenvector condition. Then $re^{ik\theta} \in \sigma(T)$ for all $k \in \bbZ$.
	\begin{proof}
		We may assume that $r = 1$. Let $\hat T$ be the operator induced by $T$ on the space $\hat E := l^\infty(\bbN;E) / c_0(\bbN;E)$. We denote by $(z_n)_{n \in \bbN}$ and $(x_n)_{n \in \bbN}$ the approximate eigenvectors from Definition~\ref{def:dae} and by $\hat z$ and $\hat x$ their equivalence classes in $\hat E$. Then we have $|\hat z| \le \hat x$ and the eigenvalue equations $\hat T \hat z = e^{i\theta}\hat z$ and $\hat T \hat x = \hat x$ hold. Hence, Proposition~\ref{prop:dominated-eigenvector}(b) implies that $e^{ik\theta} \in \sigma(\hat T) = \sigma(T)$ for all $k \in \bbZ$.
	\end{proof}
\end{theorem}

We point out that, compared to \cite[Theorem~7.4]{Gluck2015}, we did not employ a new technique for the proof of Theorem~\ref{thm:dae}; both theorems are essentially based on \cite[Theorem~3.2]{Gluck2015}. The reason why we obtained a slightly more general result here is that we made a short detour via Proposition~\ref{prop:dominated-eigenvector} in the proof.

The approximate eigenvalue $re^{i\theta}$ considered in Theorem~\ref{thm:dae} need not be a peripheral spectral value of $T$; yet, peripheral spectral values are particularly well-suited for an application of Theorem~\ref{thm:dae} since they are always approximate eigenvalues. On the other hand, it is important to note that, even in finite dimensions, there are positive operators $T$ with peripheral spectral values that do not fulfil the dominated approximate eigenvector condition; see \cite[Example~7.7]{Gluck2015} for a counterexample.

\section{Quadratic resolvent growth and the peripheral spectrum} \label{section:quadratic-resolvent-growth-the-peripheral-spectrum}

In this section we provide another condition for the peripheral spectrum of a positive operator to be cyclic. Compared to Theorem~\ref{thm:growth-rate-for-eigenvector} this condition appears to be a bit closer to what is known from the literature; the proof, however, it somewhat more involved than that of Theorem~\ref{thm:growth-rate-for-eigenvector} and uses some new techniques.

The following is the main result of this section. The reader should compare it to the known results listed in \ref{ov:cyclcity-problem-growth-conditions-on-resolvent}.

\begin{theorem} \label{thm:quadratic-growth-and-linear-growth}
	Let $E$ be a complex Banach lattice, let $T \in \calL(E)$ be a positive operator with spectral radius $r(T) = 1$ and let $\lambda \in \sigma_{\per}(T)$. If there is a sequence $1 < r_n \downarrow 1$ such that
	\begin{align*}
		\|R(r_n \lambda,T)\| \sim \frac{1}{r_n-1} \quad \text{and} \quad \|R(r_n,T)\| = O(\frac{1}{(r_n-1)^2}),
	\end{align*}
	then $\lambda^k \in \sigma_{\per}(T)$ for all $k \in \bbZ$.
\end{theorem}
	
Note that the assumption $\|R(r_n \lambda,T)\| \sim \frac{1}{r_n-1}$ in the above theorem is fulfilled if and only if $\|R(r_n \lambda,T)\| = O(\frac{1}{r_n-1})$ according to estimate (\ref{form_canonical_resolvent_estimate}) at the beginning of Section~\ref{section:overview-cycl-problem}. Before we proceed with the proof of Theorem~\ref{thm:quadratic-growth-and-linear-growth} we state a corollary and an additional remark. In case that $\lambda$ is a pole of the resolvent, the first growth condition in Theorem~\ref{thm:quadratic-growth-and-linear-growth} can be omitted:

\begin{corollary} \label{cor:quadratic-growth-and-poles-poles-of-the-resolvent}
	Let $E$ be a complex Banach lattice, let $T \in \calL(E)$ be a positive operator with spectral radius $r(T) = 1$ and let $\lambda \in \sigma_{\per}(T)$ be a pole of the resolvent $R(\phdot,T)$. If there is a sequence $1 < r_n \downarrow 1$ such that
	\begin{align*}
		\|R(r_n,T)\| = O(\frac{1}{(r_n-1)^2}),
	\end{align*}
	then $\lambda^k \in \sigma_{\per}(T)$ for all $k \in \bbZ$.
	\begin{proof}
		It follows from the resolvent estimate in the assumption and from the positivity of $T$ that the pole order of $R(\phdot,T)$ at $\lambda$ is either $1$ or $2$. If the pole order equals $1$, then the assertion follows from Theorem~\ref{thm:quadratic-growth-and-linear-growth}. If, on the other hand, the pole order equals $2$, then we have $\|R(r_n\lambda,T)\| \sim \|R(r_n,T)\|$ according to estimate (\ref{form_canonical_resolvent_estimate}) at the beginning of Section~\ref{section:overview-cycl-problem} and thus the corollary follows from assertion~\ref{ov:cyclcity-problem-growth-conditions-on-resolvent}(b).
	\end{proof}
\end{corollary}

Compared to assertion~\ref{ov:cyclicity-problem-other-conditions}(b), the novelty in Corollary~\ref{cor:quadratic-growth-and-poles-poles-of-the-resolvent} is that we assume $\lambda$ rather than $r(T)$ to be a pole of the resolvent (note however that assertion~\ref{ov:cyclicity-problem-other-conditions}(b) does not require any additional growth condition on the resolvent).

We should point out that one can also prove Corollary~\ref{cor:quadratic-growth-and-poles-poles-of-the-resolvent} without our new Theorem~\ref{thm:quadratic-growth-and-linear-growth}, only by using known results from \ref{ov:cyclcity-problem-growth-conditions-on-resolvent}. Yet, the assertion of Corollary~\ref{cor:quadratic-growth-and-poles-poles-of-the-resolvent} does not seem to have appeared in the literature, yet. Let us demonstrate how the corollary can also be derived from assertions~\ref{ov:cyclcity-problem-growth-conditions-on-resolvent}(b) and (c):

\begin{proof}[Alternative proof of Corollary~\ref{cor:quadratic-growth-and-poles-poles-of-the-resolvent}]
	Again, we observe that $\lambda$ can only be a pole of order one or two and in the latter case, the corollary follows from assertion~\ref{ov:cyclcity-problem-growth-conditions-on-resolvent}(b). We may thus assume that $\lambda$ is a first order pole of the resolvent.

	Let $P$ denote the spectral projection of $T$ associated with the isolated spectral value $\lambda$ and choose an eigenvector $0 \not= z \in \ker(\lambda - T)$; then we have $z \in PE$. Also choose a functional $z' \in E'$ such that $\langle z', z \rangle \not= 0$. Since $\lambda$ is also a first order pole of the resolvent of the adjoint operator $T'$ and since the corresponding spectral projection is given by $P'$, the range $P'E'$ coincides with $\ker(\lambda-T')$; in particular, $T' \, P'z' = \lambda \, P'z'$ and this implies that $x' := |P'z'|$ fulfils $0 \le x' \le T'x'$. Moreover, using that $Pz=z$, we obtain
	\begin{align*}
		0 < |\langle z',z\rangle| = |\langle P'z',z\rangle| \le \langle x', |z| \rangle.
	\end{align*}
	Hence, assertion~\ref{ov:cyclcity-problem-growth-conditions-on-resolvent}(c) ensures that $\lambda^k \in \sigma(T)$ for all $k\in \bbZ$.
\end{proof}

Let us recall a simple criterion which ensures that the quadratic resolvent growth condition $\|R(r_n,T)\| = O(\frac{1}{(r_n-1)^2})$ in Theorem \ref{thm:quadratic-growth-and-linear-growth} and Corollary~\ref{cor:quadratic-growth-and-poles-poles-of-the-resolvent} is fulfilled:

\begin{remark}
	Let $E$ be a complex Banach space and and let $T \in \calL(E)$ be an operator with spectral radius $r(T) = 1$. If $M := \sup_{n \in \bbN} \frac{\|T^n\|}{n} < \infty$, then
	\begin{align*}
		\|R(r,T)\| = O(\frac{1}{(r-1)^2}) \quad \text{as} \quad r \downarrow 1.
	\end{align*}
	\begin{proof}
		This follows easily from the Neumann series representation of the resolvent.
	\end{proof}
\end{remark}

The remainder of this section is devoted to the proof of Theorem~\ref{thm:quadratic-growth-and-linear-growth}. We start by providing a few ingredients for the proof, the first of which is the following lemma.

\begin{lemma} \label{lem:ideal}
	Let $E$ be a complex Banach lattice and let $(S_n) \subset \calL(E)$ be a (not necessarily norm bounded) sequence of positive operators on $E$.
	\begin{itemize}
		\item[(a)] The set
			\begin{align*}
				I := \{x \in E| \, \exists (x_n) \subset E: \; x_n \to x \text{ and } S_n|x_n| \to 0\}
			\end{align*}
			is a closed ideal in $E$. 
		\item[(b)] If $T \in \calL(E)$ is a positive operator which commutes with all operators $S_n$, then $I$ is $T$-invariant.
	\end{itemize}
	\begin{proof}
		(a) Obviously $I$ is a vector subspace of $E$. Let us prove that it is even an ideal: If $x \in I$ and $(x_n) \subset E$ is a sequence which fulfils $x_n \to x$ and $S_n|x_n| \to 0$, then we have $|x_n| \to |x|$ and $S_n\big| | x_n | \big| = S_n|x_n| \to 0$, so $|x|$ is also contained in $I$. Now, assume $0 \le y \le x \in I$ and let $(x_n)$ be as before. Then we first observe that $(\re x_n)^+ \to (\re x)^+ = x$. Now, define $y_n := y \land (\re x_n)^+$; we obtain $y_n \to y \land x = y$ and $0 \le S_n|y_n| = S_ny_n \le S_n(\re x_n)^+ \le S_n|x_n| \to 0$. Hence, $y \in I$ and we have thus proved that $I$ is an ideal. 
		
		To see that $I$ is closed, let $(x^{(k)})$ be a sequence in $I$ which converges to a vector $x \in E$. For each index $k$ there is a sequence $(x_n^{(k)})$ which converges to $x^{(k)}$ as $n \to \infty$ and which fulfils $S_n|x_n^{(k)}| \to 0$ as $n \to \infty$. Now we can find a sequence of indices $(k_l)_{l \in \bbN}$ and a strictly increasing sequence of indices $(n_l)_{l \in \bbN}$ with the following properties:
		\begin{itemize}
			\item[(i)] $\|x-x^{(k_l)}\| \le \frac{1}{2l}$ for each $l \in \bbN$. 
			\item[(ii)] $\|x^{(k_l)} - x_n^{(k_l)}\| \le \frac{1}{2l}$ for each $l \in \bbN$ and each $n \ge n_l$.
			\item[(iii)] $\|S_n\, |x_n^{(k_l)}|\, \| \le \frac{1}{l}$ for each $l \in \bbN$ and each $n \ge n_l$.
		\end{itemize}
		Let us define a sequence $(x_n)_{n \in \bbN}$ in the following way: We set $x_n = 0$ for $1 \le n \le n_1-1$ and $x_n = x_n^{(k_l)}$ for $n_l \le n \le n_{l+1}-1$. We have to show that $x_n \to x$ and that $S_n|x_n| \to 0$, so let $\varepsilon > 0$. Choose $l_0 \in \bbN$ such that $\frac{1}{l_0} < \varepsilon$ and let $n \ge n_{l_0}$. Then we can find an integer $l \ge l_0$ such that $n_l \le n \le n_{l+1}-1$. We therefore conclude from (i) and (ii) that
		\begin{align*}
			\|x - x_n\| = \|x - x_n^{(k_l)}\| \le \|x-x^{(k_l)}\| + \|x^{(k_l)} - x_n^{(k_l)}\| \le \frac{1}{l} \le \frac{1}{l_0} < \varepsilon
		\end{align*}
		and  we conclude from (iii) that
		\begin{align*}
			\|S_n \, |x_n| \,\| = \|S_n \, |x_n^{(k_l)}| \, \| \le \frac{1}{l} \le \frac{1}{l_0} < \varepsilon.
		\end{align*}
		This proves that $x_n \to x$ and $S_n|x_n| \to 0$, so $x$ is indeed contained in $I$.
		
		(b) Assertion (b) is obvious.
	\end{proof}
\end{lemma}

A few remarks on Lemma~\ref{lem:ideal} are in order:

\begin{remarks}
	Let $E$ and $(S_n)$ be as in Lemma~\ref{lem:ideal}.
	
	(a) The set $J := \{x \in E| \, S_n|x| \to 0\}$ is also an ideal in $E$. However, if the operator sequence $(S_n)$ is not norm bounded in $\calL(E)$, then $J$ might not be closed, in general.
	
	(b) Assume now that $(S_n)$ is norm bounded. Then the ideal $J$ from remark (a) is not only closed, but it also coincides with the ideal $I$ from Lemma~\ref{lem:ideal}.
	
	(c) If $(S_n)$ is not norm bounded, then we clearly have $J \subset \overline{J} \subset I$. The author does not know whether $I$ coincides with $\overline{J}$, in general. 
\end{remarks}

The other ingredient that we need for the proof of Theorem~\ref{thm:quadratic-growth-and-linear-growth} is the following simple observation about the closure of principal ideals.

\begin{proposition} \label{prop:closure-of-a-principal-ideal}
	Let $E$ be a real or complex Banach lattice, let $0 \le x,y \in E$. Then the following assertions are equivalent:
	\begin{itemize}
		\item[(i)] $x \in \overline{E_y}$.
		\item[(ii)] $\|(y-sx)^-\| = o(s)$ as $s \downarrow 0$.
	\end{itemize}
	\begin{proof}
		``(i) $\Rightarrow$ (ii)'' Assume that $x \in \overline{E_y}$ and let $\varepsilon > 0$. Since the principal ideal $E_y$ is dense in $E$ we can find a number $c > 0$ and an element $z \in E$ such that $|z| \le c y$ and such that $\|z-x\| < \varepsilon$. In particular, $\| |z| - x \| <\varepsilon$. For all sufficiently small $s > 0$ we have $c \le \frac{1}{s}$ and hence $|z| \le \frac{y}{s}$; we thus obtain
		\begin{align*}
			\frac{\|(y-sx)^-\|}{s} = \|(\frac{y}{s} - x)^-\| \le \|(\frac{y}{s} - |z|)^-\| + \| |z| - x\| < \varepsilon
		\end{align*}
		for all sufficiently small $s > 0$. This proves that $\frac{\|(y-sx)^-\|}{s} \to 0$ as $s \downarrow 0$.
		
		``(ii) $\Rightarrow$ (i)'' Assume that (ii) holds and let $\varepsilon > 0$. For a sufficiently small $s > 0$ we have $\|(\frac{y}{s}-x)^-\| = \frac{\|(y-sx)^-\|}{s} < \varepsilon$ according to (ii). Note that we have $\frac{y}{s} \in E_y$ and $0 \le (\frac{y}{s}-x)^+ \le \frac{y}{s}$; hence, the vector $z := \frac{y}{s} - (\frac{y}{s}-x)^+$ is contained in $E_y$. Since $\|z-x\| = \|(\frac{y}{s}-x)^-\| < \varepsilon$, this proves that $x \in \overline{E_y}$.
	\end{proof}
\end{proposition}

We note in passing that Proposition~\ref{prop:closure-of-a-principal-ideal} has the following corollary which is interesting in its own right (compare also \cite[Theorem~II.6.3]{Schaefer1974}).

\begin{corollary} \label{cor:approximation-of-quasi-interior-point}
	Let $E$ be a real or complex Banach lattice and $y \in E_+$. Then the following assertions are equivalent:
	\begin{itemize}
		\item[(i)] $y$ is a quasi-interior point of $E_+$.
		\item[(ii)] For every $x \ge 0$ we have $\|(y-sx)^-\| = o(s)$ as $s \downarrow 0$.
	\end{itemize}
\end{corollary}

We are now ready to prove Theorem~\ref{thm:quadratic-growth-and-linear-growth}:

\begin{proof}[Proof of Theorem~\ref{thm:quadratic-growth-and-linear-growth}]
	Replacing $E$ by an ultra power we may assume that the peripheral spectrum of $T$ consists of eigenvalues. Throughout the proof, let $z$ be an eigenvector of $T$ for the eigenvalue $\lambda$ and observe that $T|z| \ge |z|$. We divide the proof into several steps.
	
	\emph{Step 1. We construct a $T$-invariant, closed ideal $F \subset E$:} To this end, abbreviate $S := R(2,T)$, define the vector $y := S(T|z| - |z|) \ge 0$ and observe that the principal ideal $E_y$ is $T$-invariant (more generally, it is easily seen that the principal ideal $E_{Sx}$ is $T$-invariant for every $x \ge 0$). We set $F := \overline{E_y}$, and this is of course a $T$-invariant ideal, too. Observe that $F$ is also $S$-invariant by Proposition~\ref{prop:spectrum-of-induced-operators}(b).
	
	\emph{Step 2. We show that the conclusion of our theorem holds if $z \not\in F$; this is done by reducing $T$ modulo $F$:} Assume that $z \not\in F$. Denote by $T_/$ the operator induced by $T$ on the quotient space $E/F$ and observe that the two vectors $z+F$ and $S|z| +F \ge |z| + F$ are non-zero elements of $E/F$. Therefore, the first vector $z+F$ is an eigenvector of $T_/$ for the eigenvalue $\lambda$ and the second vector $S|z|+F$ is an eigenvector of $T_/$ for the eigenvalue $1$, since $TS|z| - S|z| = y \in F$. For every $k \in \bbZ$ we thus obtain $\lambda^k \in \sigma(T_/)$ by Proposition~\ref{prop:dominated-eigenvector}(b) and hence $\lambda^k \in \sigma(T)$ by Proposition~\ref{prop:spectrum-of-induced-operators}(c).
	
	Due to Step 2 we may assume for the rest of the proof that $z \in F$.
	
	\emph{Step 3. We introduce another $T$-invariant closed ideal $I \subset E$.} More precisely, we define
	\begin{align*}
		I := \{x \in E| \, \exists (x_n) \subset E: \; x_n \to x \text{ and } (r_n-1)R(r_n,T)|x_n| \to 0\}.
	\end{align*}
	By Lemma~\ref{lem:ideal} $I$ is indeed a closed, $T$-invariant ideal in $E$. 
	
	\emph{Step 4. We show that $z \not \in I$:} To this end, consider an arbitrary sequence $(x_n) \subset E$ with $x_n \to z$. Then we obtain the estimate
	\begin{align*}
		(r_n-1) R(r_n,T)|x_n| & \ge (r_n-1)|R(r_n\lambda,T)x_n| \ge \\
		& \ge (r_n-1) |R(r_n\lambda,T)z| - (r_n-1)|R(r_n\lambda,T)(z-x_n)| = \\
		& = |z| - (r_n-1)|R(r_n\lambda,T)(z-x_n)| \to |z| \not= 0;
	\end{align*}
	in the last line we used that the sequence $(r_n-1)\|R(r_n\lambda,T)\|$ is bounded by assumption. Hence, $(r_n-1)R(r_n,T)|x_n| \not\to 0$ and we conclude that $z \not\in I$ as claimed. 
	
	\emph{Step 5. We show that $y \in I$:} Let $y_n := S(T-r_n)|z|$. Then we have
	\begin{align*}
		y_n = ST|z| - r_nS|z| = y - (r_n-1)S|z| \to y.
	\end{align*}
	Since $y_n \to y$, we only have to show that $(r_n-1)R(r_n,T)|y_n| \to 0$. To obtain this, we first observe that $|y_n| = y_n + 2y_n^-$ and that $\|y_n^-\| = o(r_n-1)$. Indeed, according to Step 2 we may assume that $z \in F$ and hence that $S|z| \in F = \overline{E_y}$; this implies $y_n^- = o(r_n-1)$ by Proposition~\ref{prop:closure-of-a-principal-ideal}. We now compute
	\begin{align*}
		(r_n-1)R(r_n,T)|y_n| & = (r_n-1)R(r_n,T) S(T-r_n)|z| + 2(r_n-1)R(r_n,T)y_n^- = \\
		& = -(r_n-1)S|z| + 2(r_n-1)R(r_n,T)y_n^-.
	\end{align*}
	The first part of the latter expression clearly converges to $0$, and for the second part we obtain
	\begin{align*}
		2(r_n-1)R(r_n,T)y_n^- = 2(r_n-1)^2 R(r_n,T)\frac{y_n^-}{r_n-1} \to 0
	\end{align*}
	since $\|R(r_n,T)\| = O(\frac{1}{(r_n-1)^2})$ by assumption and since $\|y_n^-\| = o(r_n-1)$. Hence, $y \in I$.
	
	\emph{Step 6. We prove the conclusion of our theorem by reducing $T$ modulo $I$:} Let $T_/$ be the operator induced by $T$ on the quotient space $E/I$. According to Step 4, the two vectors $z+I$ and $S|z| + I \ge |z| + I$ are non-zero elements of $E/I$. The first vector $z +I$ is thus an eigenvector of $T_/$ for the eigenvalue $\lambda$; the second vector $S|z| + I$ is an eigenvector of $T_/$ for the eigenvalue $1$, since $TS|z|-S|z| = y \in I$ according to Step 5.
	
	For every $k \in \bbZ$ we thus conclude that $\lambda^k \in \sigma(T_/)$ by Proposition~\ref{prop:dominated-eigenvector}(b) and hence, $\lambda^k \in \sigma(T)$ by Proposition~\ref{prop:spectrum-of-induced-operators}(c).
\end{proof}

It is worthwhile pointing out that the above proof would become much simpler if we even assumed that $\|R(r_n,T)\| = o(\frac{1}{(r_n-1)^2})$. Let us sketch briefly how this works:

\begin{proof}[Simplified proof of Theorem~\ref{thm:quadratic-growth-and-linear-growth} under an additional assumption.] Suppose that, in addition to the assumptions of Theorem~\ref{thm:quadratic-growth-and-linear-growth}, we even have $\|R(r_n,T)\| = o(\frac{1}{(r_n-1)^2})$. As in the original proof, let $z \in E$ be an eigenvector of $T$ for the eigenvalue $\lambda$.

	\emph{Steps 1 and 2}. Steps 1 and 2 of the original proof are no longer needed.
	
	\emph{Steps 3 and 4}. Steps 3 and 4 are the same as in the original proof.
	
	\emph{Step 5.} In contrast to the original proof one can now simply define $y := T|z| - |z|$. Using the strengthened assumption $\|R(r_n,T)\| = o(\frac{1}{(r_n-1)^2})$ one readily verifies that $(r_n-1)R(r_n,T)|y| \to 0$. Hence, $y \in I$.
	
	\emph{Step 6.} As in the original proof one considers the operator $T_/ \in \calL(E/I)$ which is induced by $T$. By Steps 4 and 5, $z+I$ and $|z|+I$ are eigenvectors of $T_/$ for the eigenvalues $\lambda$ and $1$, respectively. Hence, $\lambda^k \in \sigma(T_/)$ and thus $\lambda^k \in \sigma(T)$ for each $k \in \bbZ$ (due to Propositions~\ref{prop:dominated-eigenvector}(b) and~\ref{prop:spectrum-of-induced-operators}(c)).
\end{proof}

\section{Invariant ideals for (WS)-bounded operators} \label{section:invariant-ideals-for-ws-bounded-operators}

The aim of this final section is twofold: First, we want to give a new proof of \cite[Theorem~7.1]{Gluck2015} which asserts that every so-called \emph{(WS)-bounded} positive operator on a Banach lattice has cyclic peripheral spectrum. Second, we want generalize some known results about Abel bounded, irreducible operators to (WS)-bounded, irreducible operators. Our approach to both aims is based on an auxiliary result about invariant ideals of (WS)-bounded operators in Proposition~\ref{prop:invariant-ideal-for-ws-bounded-operators} below. Before we state and prove this proposition, we recall from \cite[Section~4]{Gluck2015} the definition of (WS)-bounded operators and the underlying notion of \emph{weighting schemes}.

Let $\overline{\bbD} \subset \bbC$ be the closed unit disk; a function $f: \overline{\bbD} \to \bbC$ is called \emph{analytic} if it has an analytic extension to some open neighbourhood of $\overline{\bbD}$. The following definition stems from \cite[Definition~4.1]{Gluck2015}.

\begin{definition} \label{def:weighting-scheme}
	A net $(f_j)_{j \in J}$  of analytic functions $f_j: \overline{\bbD} \to \bbC$ is called a \emph{weighting scheme} if the following three conditions are fulfilled:
	\begin{description}
		\item[(WS1)] $f_j(1) = 1$ for all $j \in J$.
		\item[(WS2)] $f_j^{(k)}(0) \ge 0$ for all $j \in J$ and all $k \in \bbN_0$. 
		\item[(WS3)] $|f_j(z)| \overset{j}{\to} 0$ for all $z \in \bbC$ with $|z| < 1$.
	\end{description}
\end{definition}

Note that every subnet of a weighting scheme is itself a weighting scheme. Typical examples of weighting schemes are $f_j(z) = z^j$ for $j \in \bbN_0$, $f_j(z) = \frac{1}{j+1}\sum_{k=0}^jz^k$ for $j \in \bbN_0$ and $f_r(z) = \frac{r-1}{r-z}$, where $r \in (1,\infty)$ and where the index set $(1,\infty)$ is ordered oppositely to the usual order induced by $\bbR$. Some further examples can be found in \cite[Example~4.6]{Gluck2015}.

Whether a net of analytic functions $(f_j)_{j \in J}$ on $\overline{\bbD}$ is a weighting scheme can also be seen by considering the coefficients of the power series expansion of each function $f_j$ around $0$; this is the content of the following proposition.

\begin{proposition} \label{prop:characterization-of-weighting-schemes}
	Let $(f_j)_{j \in J}$ be a net of analytic functions $\overline{\bbD} \to \bbC$. For each $j \in J$, let $f_j(z) = \sum_{k=0}^\infty a_{j,k} z^k$ be the power series expansion of $f_j$ around $0$. Then $(f_j)_{j \in J}$ is a weighting scheme if and only if the following three conditions are fulfilled:
	\begin{itemize}
		\item[(a)] $\sum_{k=0}^\infty a_{j,k} = 1$ for each $j \in J$.
		\item[(b)] $a_{j,k} \ge 0$ for all $j \in J$ and all $k \in \bbN_0$.
		\item[(c)] For each fixed $k \in \bbN_0$ we have $a_{j,k} \overset{j}{\to} 0$.
	\end{itemize}
	\begin{proof}
		Obviously, (WS1) is equivalent to (a) and (WS2) is equivalent to (b). Now suppose that (WS1) and (WS2) (equivalently: (a) and (b)) are fulfilled: Then it was shown in \cite[Proposition~4.4]{Gluck2015} that (WS3) is equivalent to (c).
	\end{proof}
\end{proposition}

If $(f_j)_{j \in J}$ is a weighting scheme, $E$ is a complex Banach space and $T \in \calL(E)$ has spectral radius $r(T) = 1$, then the operators $f_j(T)$ ($j \in J$) are well-defined by means of the analytic functional calculus. This simple observation is used in the following definition which is taken from \cite[Definition~4.7]{Gluck2015}.

\begin{definition} \label{def:ws-bounded}
	Let $E$ be a complex Banach space and let $T \in \calL(E)$, $r(T) = 1$. The operator $T$ is called \emph{(WS)-bounded} if there is a weighting scheme $(f_j)_{j \in J}$ such that the set $\{f_j(T)| \, j \in J\} \subset \calL(E)$ is bounded with respect to the operator norm.
\end{definition}

If we consider the examples of weighting schemes which were given right after Definition~\ref{def:weighting-scheme} and take into account that a subnet of a weighting scheme is again a weighting scheme, then we obtain immediately that every Ces\`{a}ro-bounded, more generally every Abel-bounded, and every partially power bounded operator is (WS)-bounded. Several further classes of (WS)-bounded operators are listed in \cite[Example~4.8]{Gluck2015}. Of course it would also be possible to define a similar notion for operators $T$ with spectral radius $r(T) \in (0,\infty)$ by a rescaling (see also the discussion after \cite[Definition~4.7]{Gluck2015}). However, we find it more comprehensible to state all our results for operators with spectral radius $1$; the reader can then easily obtain the general result by a rescaling argument. 

We are now ready to state the following proposition on which the rest of the section is based.

\begin{proposition} \label{prop:invariant-ideal-for-ws-bounded-operators}
	Let $E$ be a complex Banach lattice and let $T\in \calL(E)$ with $r(T) = 1$ be positive and (WS)-bounded. Then there exists a closed ideal $I \subset E$ with the following properties:
	\begin{itemize}
		\item[(a)] $I$ is $T$-invariant.
		\item[(b)] For each eigenvector $z \in E$ of $T$ which belongs to a peripheral eigenvalue we have $z \not\in I$, but $T|z| - |z| \in I$.
		\item[(c)] Whenever $x \in E$ and $T^k|x| \to 0$ as $k \to \infty$, then $x \in I$.
	\end{itemize}
	\begin{proof}
		By assumption, there exists a weighting scheme $(f_j)_{j \in J}$ for which the set $\{f_j(T)|\, j \in J\} \subset \calL(E)$ is bounded in operator norm. We define the wanted ideal $I$ by
		\begin{align*}
			I := \{y \in E| \, f_j(T)|y| \overset{j}{\to} 0 \text{ weakly}\}.
		\end{align*}
		This is easily seen to be indeed an ideal in $E$ and since $\{f_j(T)| \, j \in J\}$ is operator norm bounded, $I$ is also closed. Let us show that $I$ fulfils the claimed properties (a)--(c):
		
		(a) Since $T$ is weakly continuous and $f_j(T)$ commutes with $T$ for every $j \in J$, $I$ is $T$-invariant. 
		
		(b) Let $z \in E$ be an eigenvector of $T$ belonging to a peripheral eigenvalue. Note that we have $T|z| \ge |z|$ and thus $f_j(T)|z| \ge |z|$ for each $j \in J$. Hence, $z \not\in I$. To show that $T|z| - |z| \in I$, we first make the following general observation:
		\begin{center}
		$(*)$ Consider $E$ as a sublattice of its bi-dual $E''$ by means of evaluation and suppose that $0 \le x \le Tx$ for some $x \in E$. Then the nets $(f_j(T)x)_{j \in J}$ and $(T^kx)_{k \in \bbN_0}$ are both weak${}^*$-convergent in $E''$ and their limits coincide. 
		\end{center}
		
		To prove $(*)$ first note that the sequence $(T^kx)_{k \ge 0}$ is positive and increasing, i.e.~$0 \le x \le Tx \le T^2x \le ...$; by \cite[Lemma~4.13]{Gluck2015} it is thus norm bounded and hence, it weak${}^*$-converges in $E''$ to some element $x'' \in E''$. Using this together with Proposition~\ref{prop:characterization-of-weighting-schemes} one immediately checks that $(f_j(T)y)_{j \in J}$ weak${}^*$-converges to $x''$, too. Hence, $(*)$ holds.
		
		We now apply $(*)$ once to the vector $x = |z|$ and once to the vector $x = T|z|$. Clearly, the weak${}^*$-limits of the sequences $(T^k|z|)_{k \in \bbN_0}$ and $(T^kT|z|)_{k \in \bbN_0}$ in $E''$ coincide and by $(*)$ this implies that the nets $(f_j(T)|z|)_{j \in J}$ and $(f_j(T)T|z|)_{j \in J}$ also have the same weak${}^*$-limit in $E''$. Hence, $f_j(T)|T|z| - |z| \,| = f_j(T)T|z| - f_j(T)|z|$ weak${}^*$-converges to $0$ in $E''$ which is equivalent to saying that it weakly converges to $0$ in $E$. Thus, $T|z| - |z| \in I$ as claimed.
		
		(c) Now, let $x \in E$ and suppose that $T^k|x| \to 0$ as $k \to \infty$. Then one easily checks that $f_j(T)|x| \to 0$ even with respect to the norm on $E$. In particular, $x \in I$.
	\end{proof}
\end{proposition}

In \cite[Theorem~7.1]{Gluck2015} it was shown by the author that every (WS)-bounded positive operator on a Banach lattice has cyclic peripheral spectrum; it was also pointed out there that the author did not know whether this theorem can be proved by a (modification of a) classical method of Lotz from \cite[the proof of Theorem~4.7]{Lotz1968}. Using Proposition~\ref{prop:invariant-ideal-for-ws-bounded-operators} we can now indeed give a new proof of \cite[Theorem~7.1]{Gluck2015} by a technique similar to Lotz':

\begin{corollary} \label{cor:ws-bounded-positive-operators-has-cyclic-peripheral-spectrum}
	Let $E$ be a complex Banach lattice, let $T \in \calL(E)$ with $r(T) = 1$ be positive and (WS)-bounded. Then the peripheral spectrum of $T$ is cyclic.
	\begin{proof}
		Embedding $E$ into an ultra power we may assume that $\sigma_{\per}(T)$ consists of eigenvalues. Let $\lambda$ be a peripheral eigenvalue of $T$ and $z \in E$ a corresponding eigenvector. Let $I$ be the $T$-invariant ideal given in Proposition~\ref{prop:invariant-ideal-for-ws-bounded-operators}. If $T_/$ denotes the operator induced by $T$ on the quotient space $E/I$, then we obtain, using property (b) in Proposition~\ref{prop:invariant-ideal-for-ws-bounded-operators}, that $z+ I$ and $|z|+I$ are non-zero elements of $E/I$ and eigenvectors of $T_/$ for the eigenvalues $\lambda$ and $1$, respectively. Proposition~\ref{prop:dominated-eigenvector}(b) thus shows that $\lambda^k \in \sigma(T_/)$ for every $k \in \bbZ$. Hence, $\lambda^k \in \sigma(T)$ for every $k \in \bbZ$ by Proposition~\ref{prop:spectrum-of-induced-operators}(c).
	\end{proof}
\end{corollary}

For irreducible positive operators much more on the peripheral (point) spectrum can be said than for arbitrary positive operators. For example, let $E$ be a complex Banach lattice, let $T \in \calL(E)$ with $r(T) = 1$ be positive, irreducible and Abel bounded and assume that the peripheral point spectrum of $T$ is non-empty. If $z \in E$ is an eigenvector for some peripheral eigenvalue of $T$, then $|z| \le T|z|$ and it thus follows from \cite[Lemma~V.4.8]{Schaefer1974} that the adjoint $T'$ possesses a fixed vector $x' > 0$. Since $T$ is irreducible and since the set $\{x \in E|\, \langle x', |x| \rangle = 0\}$ is a $T$-invariant closed  and proper ideal in $E$ and thus equals $\{0\}$, we obtain that $x'$ is even \emph{strictly positive}, meaning that we have $\langle x', x \rangle > 0$ for each $0 < x  \in E$. One can then employ \cite[Theorem~V.5.2]{Schaefer1974} or \cite[Theorem~4.12]{Grobler1995} to obtain very precise information on the spectral properties of $T$. Now, assume that $T$ is no longer Abel bounded, but only (WS)-bounded. Then it is no longer clear (at least not to the author) whether $T'$ possesses a positive non-zero fixed vector. However, we can still prove the same conclusions as in \cite[Theorem~V.5.2]{Schaefer1974} or \cite[Theorem~4.12]{Grobler1995}. 

To state those results in the next theorem, recall that two bounded linear operators $S$ and $T$ on a Banach space $E$ are called \emph{similar} if there is a bijection $V \in \calL(E)$ such that $T = V^{-1}SV$. 

\begin{theorem} \label{thm:irreducible}
	Let $E$ be a complex Banach lattice and let $T \in \calL(E)$ with $r(T) = 1$ be positive and irreducible. Assume that the peripheral point spectrum of $T$ is non-empty and that $T$ is (WS)-bounded. Then the following assertions hold:
	\begin{itemize}
		\item[(a)] For every $0 < x \in E$ we have $T^kx \not\to 0$ as $k \to \infty$.
		\item[(b)] Whenever $z \in \ker(\lambda-T)$ for a peripheral eigenvalue $\lambda$ of $T$, then $|z| \in \ker(1-T)$.
		\item[(c)] $1$ is an eigenvalue of $T$ and the corresponding eigenspace $\ker(1-T)$ is a one-dimensional sublattice of $E$ which is spanned by a quasi-interior point of $E_+$.
		\item[(d)] For each peripheral eigenvalue $\lambda$ of $T$ the operators $T$ and $\lambda T$ are similar. In particular we have $\lambda\sigma(T) = \sigma(T)$ and $\lambda\sigma_{\pnt}(T) = \sigma_{\pnt}(T)$.
		\item[(e)] The peripheral point spectrum of $T$ is a subgroup of the complex unit circle.
		\item[(f)] Each peripheral eigenvalue $\lambda$ of $T$ is an algebraically simple eigenvalue of $T$, i.e.~we have $\dim \bigcup_{n \in \bbN} \ker((\lambda-T)^n) = \dim \ker(\lambda-T) = 1$.
		\item[(g)] The only eigenvalue of $T$ which admits a positive eigenvector is $1$.
	\end{itemize}
\end{theorem}

The proof of Theorem~\ref{thm:irreducible} is in some respects very similar to the proofs of \cite[Theorem~V.5.2]{Schaefer1974} or \cite[Theorem~4.12]{Grobler1995}. However, in some points our proof differs since we do no longer know whether we have a strictly positive fixed point of $T'$ at our disposal. 

To give the proof of Theorem~\ref{thm:irreducible} we first need a number of ingredients. Recall that an operator $U \in \calL(E)$ on a complex Banach lattice $E$ is called a \emph{torsion operator} if we have $|Ux| = |x|$ for each $x \in E$. We observe the following fact:

\begin{remark}
	Let $E$ be a complex Banach lattice and let $U \in \calL(E)$ be a torsion operator. Then $U$ is bijective and its inverse $U^{-1}$ is a torsion operator, too.
	\begin{proof}
		Observe that if $U$ is bijective, then $U^{-1}$ is obviously a torsion operator, too. Moreover, $U$ is clearly isometric and thus injective; it therefore remains to prove that $U$ is surjective. We do this in two steps:
		
		\emph{Step 1. We first assume that $E$ is a $C(K;\bbC)$-space for some compact Hausdorff space $K \not= \emptyset$.} Then $U$ is contained in the so-called \emph{center} of $\calL(E)$ and thus it is a multiplication operator with symbol $u \in C(K;\bbC)$ (see e.g.~\cite[Section~C-I-9]{Arendt1986}). We clearly have $|u| = |u\mathbbm{1}_K| = |U\mathbbm{1}_K| = \mathbbm{1}_K$ and thus $U$ is obviously bijective.
		
		\emph{Step 2: Now, let $E$ be arbitrary.} If $x > 0$, then the principal ideal $E_x$ generated by $x$ is invariant with respect to $U$; moreover, it is an AM-space with unit $x$ when endowed with an appropriate norm (see \cite[the corollary of Proposition~II.7.2]{Schaefer1974}) and may thus be identified with some $C(K;\bbC)$-space. The restriction of $U$ to $E_x$ is clearly a torsion operator on the AM-space $E_x$ and thus surjective onto $E_x$ by Step 1. Since the restriction of $U$ to any non-zero principal ideal is surjective onto this principal ideal, $U$ must be surjective itself.
	\end{proof}
\end{remark}

Another ingredient for the proof of Theorem~\ref{thm:irreducible} is the following well-known auxiliary result.

\begin{lemma} \label{lem:similarity-lemma}
	Let $E = C(K;\bbC)$ for some compact Hausdorff space $K \not= \emptyset$ and let $T \in \calL(E)$ be a Markov operator on $E$. If $\lambda$ is a peripheral eigenvalue of $E$ and if $f \in C(K;\bbC)$ is a corresponding eigenfunction and satisfies $|f| = \mathbbm{1}_K$, then there is a torsion operator $U \in \calL(E)$ such that $\lambda T = U^{-1}TU$.
	\begin{proof}
		See \cite[Theorem~B-III-2.4(a)]{Arendt1986}.
	\end{proof}
\end{lemma}

We also need the following well-known fact about Banach lattices:

\begin{remark} \label{rem:one-dimensional-lattice}
	Let $E \not= \{0\}$ be a real Banach lattice and suppose that every element $0 < x \in E$ is a quasi-interior point of $E_+$. Then $E$ is one-dimensional.
\end{remark}

A proof of this well-known result can for example be found in \cite[Lemma~5.1]{Lotz1968}; since this reference is however written in German, we include here an English translation of the proof for the convenience of the reader.

\begin{proof}[Proof of Remark~\ref{rem:one-dimensional-lattice}]
	By \cite[Propositions~II.3.4 and~II.5.2(ii)]{Schaefer1974} we only have to prove that $E$ is totally ordered. To do so, let $x,y \in E$. Then $x - x \land y$ and $y - x \land y$ are positive vectors and their infimum equals $0$. It thus follows from the assumption that one	of them is $0$, which proves that $x \le y$ or $y \le x$.
\end{proof}

We can now give the proof of Theorem~\ref{thm:irreducible}.

\begin{proof}[Proof of Theorem~\ref{thm:irreducible}]
	(a) and (b) Let $I \subset E$ be the ideal given by Proposition~\ref{prop:invariant-ideal-for-ws-bounded-operators}. Since the peripheral point spectrum of $T$ is non-empty by assumption we conclude from assertion (b) of Proposition~\ref{prop:invariant-ideal-for-ws-bounded-operators} that $I \not= E$. As $T$ is irreducible and $I$ is $T$-invariant it follows that $I = \{0\}$. 
	
	Now assume that $x \in E_+$ and $T^kx \to 0$. Then $x \in I$ by Proposition~\ref{prop:invariant-ideal-for-ws-bounded-operators}(c) and thus $x = 0$. This proves (a). To prove (b), let $z \in E$ be an eigenvector for a peripheral eigenvalue of $T$. Then Proposition~\ref{prop:invariant-ideal-for-ws-bounded-operators}(b) yields $T|z| - |z| \in I = \{0\}$, so $|z| \in \ker(1-T)$.

	(c) As the peripheral point spectrum of $T$ is non-empty, it follows from (b) that $1$ is an eigenvalue of $T$. Moreover, since $T$ maps real elements of $E$ to real elements again, $\ker(1-T)$ is conjugation invariant, so (b) implies that $\ker(1-T)$ is a sublattice of $E$. Let us prove that $\ker(1-T)$ is one-dimensional: since $T$ is irreducible, every non-zero positive element of $\ker(1-T)$ is a quasi-interior point of $E_+$ and thus also a quasi-interior point of $E_+ \cap \ker(1-T)$ (see \cite[Corollary 2 of Theorem~II.6.3]{Schaefer1974}). By Remark~\ref{rem:one-dimensional-lattice} the real part of $\ker(1-T)$ is therefore a one-dimensional real vector space and hence $\ker(1-T)$ is a one-dimensional complex vector space. Since $\ker(1-T)$ contains a non-zero positive vector, and this vector is a quasi-interior point of $E_+$ and spans $\ker(1-T)$, assertion (c) is proved.
	
	(d) Let $\lambda$ be a peripheral eigenvalue and let $z \in E$ be a corresponding eigenvector. According to (b) and (c), $|z|$ is contained in $\ker(1-T)$ and is a quasi-interior point of $E_+$. When endowed with an appropriate norm, the principal ideal $E_{|z|}$ becomes an AM-space with unit $|z|$ (see \cite[the corollary of Proposition~II.7.2]{Schaefer1974}) and clearly $E_{|z|}$ is invariant with respect to $T$. Since $T|_{E_{|z|}}$ is a Markov operator on the AM-space $E_{|z|}$, Lemma~\ref{lem:similarity-lemma} implies that there is a torsion operator $U$ on $E_{|z|}$ such that $\lambda T|_{E_{|z|}} = U^{-1}T|_{E_{|z|}}U$. Obviously, $U$ is isometric with respect to the $E$-norm and since $E_{|z|}$ is dense in $E$, $U$ has a continuous linear extension $\tilde U\in \calL(E)$ which is again a torsion operator and which fulfils the equation $\lambda T = \tilde U^{-1} T \tilde U$. Hence, $T$ and $\lambda T$ are similar. The remaining assertions of (d) are now obvious.
	
	(e) Assertion (e) follows immediately from the assertion $\lambda \sigma_{\pnt}(T) = \sigma_{\pnt}(T)$ in (d).
	
	(f) By (d) it suffices to prove the assertion for $\lambda=1$. We already know from (c) that $1$ is a geometrically simple eigenvalue.	Next we note that $\ker((1-T)^2) = \ker(1-T)$; indeed, since $T$ is (WS)-bounded we can use exactly the same proof as in \cite[Proposition~4.10]{Gluck2015} to see that there exists no vector in $\ker((1- T)^2) \setminus \ker(1- T)$ (note that in \cite[Proposition~4.10]{Gluck2015} $r(T) = 1$ was assumed to be a pole of the resolvent; however, one does not need this assumption to show that $\ker((1- T)^2) \setminus \ker(1- T) = \emptyset$).
	
	Finally, let $n \ge 2$ and assume that $y \in \ker((1-T)^n)$. Then $(1-T)^{n-2}y \in \ker((1-T)^2) = \ker(1-T)$ and thus, $y \in \ker((1-T)^{n-1})$. We have therefore proved that $\ker((1-T)^n) = \ker((1-T)^{n-1})$ for each $n \ge 2$ and iterating this equality we obtain that $\ker((1-T)^n) = \ker(1-T)$ for each $n \ge 2$. This proves (f).
	
	(g) Let $\lambda$ be an eigenvalue of $T$ with corresponding eigenvector $z > 0$. Then $\lambda$ must be a peripheral eigenvalue since if we assumed $|\lambda| < 1$, then $T^kz \to 0$ as $k \to\infty$ which contradicts (a). However, since $\lambda$ is a peripheral eigenvalue we can employ (b) to obtain $\lambda z = Tz = T|z| = |z| = z$. Hence, $\lambda = 1$.
\end{proof}

\appendix

\section{Invariant subspaces and the spectrum of induced operators}

In this appendix we recall a few spectral results related to invariant subspaces. To state the results, we need the following notation: Let $E$ be a complex Banach space and let $F \subset E$ be a closed vector subspace. Whenever $T \in \calL(E)$ leaves $F$ invariant, then we denote by $T_| \in \calL(F)$ the restriction of $T$ to $F$ and by $T_/ \in \calL(E/F)$ the operator induced by $T$ on the quotient space $E/F$.

\begin{proposition} \label{prop:spectrum-of-induced-operators}
	Let $E$ be a complex Banach space, let $F \subset E$ be a closed vector subspace and suppose that $T \in \calL(E)$ leaves $F$ invariant. 
	\begin{itemize}
		\item[(a)] We have $\|T_|\| \le \|T\|$ and $\|T_/\| \le \|T\|$ as well as $r(T_|) \le r(T)$ and $r(T_/) \le r(T)$. 
		\item[(b)] For all $\lambda \in \bbC$ with $|\lambda| > r(T)$ the resolvent operator $R(\lambda,T)$ also leaves $F$ invariant and we have $R(\lambda,T)_| = R(\lambda,T_|)$ and $R(\lambda,T)_/ = R(\lambda, T_/)$.
		\item[(c)] If $\lambda \in \bbC$ has modulus $r(T)$ and is a spectral value of $T_|$ or of $T_/$, then it is also a spectral value of $T$.
	\end{itemize}
	\begin{proof}
		(a) and (b): One easily checks that $\|T_|\| \le \|T\|$ and $\|T_/\| \le \|T\|$. Now, let $\lambda \in \bbC$ with $|\lambda| > r(T)$. It follows from the Neumann series representation of the resolvent that $R(\lambda,T)$ leaves $F$ invariant, too. We have 
		\begin{align*}
			R(\lambda,T)_| \, (\lambda-T_|) = (\lambda - T_|) \, R(\lambda,T)_| = \operatorname{id}_{F},
		\end{align*}
		and the same holds for the operators on $E/F$. Hence, $\lambda$ is in the resolvent set of $T_|$ and $T_/$ and we have $R(\lambda,T)_| = R(\lambda,T_|)$ as well as $R(\lambda,T)_/ = R(\lambda,T_/)$. This also shows that $r(T_|) \le r(T)$ and $r(T_/) \le r(T)$.
		
		(c) Let $\lambda \in \sigma(T_|)$. Choose a sequence $(\lambda_n) \subset \bbC$ which converges to $\lambda$ and fulfils $|\lambda_n| > r(T)$ for each $n$. We then have
		\begin{align*}
			\|R(\lambda_n,T)\| \ge \|R(\lambda_n,T)_|\| = \|R(\lambda_n,T_|)\| \to \infty \quad \text{as } n \to \infty.
		\end{align*}
		Hence, $\lambda \in \sigma(T)$. If $\lambda \in \sigma(T_/)$, the same estimate works on $E/F$ instead of $F$.
	\end{proof}
\end{proposition}

%

\bibliographystyle{plain}
\bibliography{literature}

\end{document}